\documentclass{amsart}

\newtheorem{theorem}{Theorem}[section]
\newtheorem{corollary}[theorem]{Corollary}
\newtheorem{proposition}[theorem]{Proposition}

\begin{document}

\title{Non left-orderable surgeries and generalized Baumslag-Solitar relators}

\author{Kazuhiro Ichihara}
\address{Department of Mathematics, 
College of Humanities and Sciences, Nihon University,
3-25-40 Sakurajosui, Setagaya-ku, Tokyo 156-8550, Japan}
\email{ichihara@math.chs.nihon-u.ac.jp}

\author{Yuki Temma}
\address{Graduate School of Integrated Basic Sciences, Nihon University,
3-25-40 Sakurajosui, Setagaya-ku, Tokyo 156-8550, Japan}
\email{s6113M12@math.chs.nihon-u.ac.jp}

\begin{abstract}
We show that a knot has a non left-orderable surgery if the knot group admits a generalized Baumslag-Solitar relator and satisfies certain conditions on a longitude of the knot.
As an application, it is shown that certain positively twisted torus knots admit non left-orderable surgeries. 
\end{abstract}



\date{\today}

\maketitle

\section{Introduction}

In \cite[Conjecture 1]{BGW}, Boyer, Gordon and Watson proposed the following conjecture, 
now called the \textit{L-space Conjecture}: 
An irreducible rational homology 3-sphere is an L-space if and only if its fundamental group is non left-orderable.
This has now become one of the important open problems in knot theory and Dehn surgery theory. 
Here, a rational homology sphere is called an \textit{L-space} if the rank of the Heegaard Floer homology is equal to the order of the first (classical) homology group for the manifold. See \cite{PZ}.  
Also a non-trivial group $G$ is called \textit{left-orderable}, often abbreviated as LO, if there exists a strict total ordering $>$ for the elements of $G$ that is left-invariant: For any $f, g, h \in G$, whenever $g>h$ then $fg>fh$. 
We use the convention that the trivial group is not considered to be left-orderable.

On the study of the L-space Conjecture, 
one of the simple ways to construct L-spaces is given by Dehn surgery. 
For instance, any given knot in the 3-sphere $S^3$ admitting an L-space by surgery gives rise to an infinite family of L-spaces.
In fact, it was shown in \cite[Proposition 9.6]{OS} that if a knot $K$ has one L-space surgery, then all the surgered manifolds are L-spaces if the surgery slopes are greater than or equal to $2g(K)-1$, where $g(K)$ denotes the genus of the knot. 
Precisely, for a given knot $K$ in a 3-manifold $M$, 
the following operation is called \textit{Dehn surgery}; 
removing an open regular neighborhood of $K$ from $M$, and gluing a solid torus back. 
Let us denote by $K(p/q)$ the 3-manifold obtained by Dehn surgery on $K$ in $S^3$ 
along the slope $r$ (i.e., the meridian of the attached solid torus is identified with the curve of the slope $r$), which corresponding to $p/q \in \mathbb{Q}$ on the peripheral torus of $K$. 
Please refer \cite{B} on details, for instance. 

We say that a Dehn surgery on a knot is \textit{non left-orderable} 
if it yields a closed 3-manifold with the non left-orderable fundamental group. 
There are several known studies on 3-manifolds with a non left-orderable fundamental group by using Dehn surgery. 
For example, there are works by Clay and Watson \cite{CW} and Nakae \cite{N}. 
In this paper, we show the following, which gives an extension of a result of Nakae.

\begin{theorem}\label{thm}
Let $K$ be a knot in a closed, connected 3-manifold $M$. 
Suppose that the knot group $\pi_1(M-K)$ has a presentation such as 
\[
\langle a, b \ \vert \ (w_1 a^m w_1^{-1}) b^{-r} (w_2^{-1} a^n w_2) b^{r-k} \rangle
\]
Here $w_1, w_2$ are arbitrary words 
with $m, n\ge 0$, $r\in \mathbb{Z}$, $k\ge 0$. 
Suppose further that $a$ represents a meridian of $K$ and 
$a^{-s} w a^{-t}$ represents a longitude of $K$ with $s, t\in \mathbb{Z}$ and $w$ is a word which excludes $a^{-1}$ and $b^{-1}$.
Then if $q\ne 0$ and $p/q\ge s+t$, then Dehn surgery on $K$ along the slope $p/q$ yields a closed 3-manifold with non left-orderable fundamental group.
\end{theorem}

There exist knots in $S^3$ whose knot groups admit presentations satisfying the conditions in our theorem. 
For instance, as shown by Nakae in \cite[Proposition 2.1]{N}, 
the $(-2,3, 2s+1)$-pretzel knot with $s \ge 3$ gives such examples. 
Thus our theorem presents an alternative proof of the result \cite[Corollary 4.2]{N}. 

\medskip

Furthermore, in Proposition \ref{prp}, we see that 
certain positively twisted torus knots have the knot groups admitting the presentations satisfying the conditions in our theorem. 
Therefore, as an application to Theorem \ref{thm}, we show that certain positively twisted torus knots admit non left-orderable surgeries as follows. 

\begin{corollary}\label{cor}
Let $K$ be a positively $u$-twisted $(3, 3v + 2)$-torus knot with $u, v \ge 0$. 
Then if $q\ne 0$ and $p/q\ge 2 u + 3(3v+2)$, then Dehn surgery on $K$ along the slope $p/q$ yields a closed 3-manifold with non left-orderable fundamental group.
\end{corollary}

This gives an extension to the results obtained by Clay and Watson \cite[Theorems 4.5 and 4.7]{CW}. 
There they gave proofs for the cases (1) $u \ge 0$ and $v=1$, (2) $u=1$ and $v \ge0$. 
Also note that Vafaee showed in \cite{V} that such twisted torus knots admit L-space surgery. 
Thus the corollary also gives a supporting evidence for the L-space conjecture. 

\medskip

We here remark that the relator in the presentation of $\pi_1(M-K)$ above can be regarded as a generalization of the well-known Baumslag-Solitar relator. 
By the \textit{Baumslag-Solitar relator}, we mean the relator $x^{-n} y x^m y^{-1}$ with non-zero integers $m$ and $n$ in the group generated by two elements $x$ and $y$. 
The groups with two generators and the Baumslag-Solitar relator was originally introduced in \cite{BS}, now called a \textit{Baumslag-Solitar group}, which plays an important role and is well-studied in combinatorial group theory and geometric group theory. 

In particular, in \cite{S}, it was shown that the Baumslag-Solitar relator cannot appear in a non-degenerate way in the fundamental group of an orientable 3-manifold. 
Our relator is obtained from the Baumslag-Solitar relator by replacing $x$ with some conjugates of it, and so it can be regarded as a generalization of the relator. 
It thus seems interesting that our relator can actually appear in the knot groups for certain knots in $S^3$, or more precisely, for the $(-2,3,n)$-pretzel knots, and can play an essential role to study (non) left-orderability of the groups.

\section{Proof}

The following is the key proposition to prove the main theorem. 

A homomorphism $\Phi$ of a group $G$ to $Homeo^+(\mathbb{R})$, i.e., the group of order-preserving homeomorphisms on $\mathbb{R}$, has a \textit{global fixed point} if there exists $x$ in $\mathbb{R}$ such that $\Phi(g) x=x$ for all $g$ in $G$. 
In the following, we will confuse, by abuse of notation $g$ and $\Phi(g)$ for an element $g$ in $G$.

\begin{proposition}
Suppose that a group $G$ has a presentation such as 
\[
\langle a, b \ \vert \ ( w_1a^m w_1^{-1} ) b^{-r} ( w_2^{-1} a^n w_2 ) b^{r-k}, MLM^{-1}L^{-1}, M^pL^q \rangle
\]
Here $w_1, w_2$ are arbitrary words with $m, n\ge 0$, $r\in \mathbb{Z}$, $k\ge 0$, $p, q\in \mathbb{Z}$, $M=a$, $L=a^{-s} w a^{-t}$, $w$ is a word which contains at least one $b$ and excludes $a^{-1}$ and $b^{-1}$, $s, t\in \mathbb{Z}$.
If $q\ne 0$ and $p/q\ge s+t$, then every homomorphism $\Phi: G \to Homeo^+(\mathbb{R})$ 
has a global fixed point.
\end{proposition}

\begin{proof}
We first consider the case where $a$ has a fixed point, say $x$, on $\mathbb{R}$. 
In this case, we can show that $bx=x$ as follows. 
If $b$ could not fix $x$, we can assume $bx>x$ for the $x$ by choosing an order on $\mathbb{R}$. 
It is equivalent to $b^{-1} x<x$ since any action of $G$ preserves an order of $\mathbb{R}$. 
Also $ax=x$, for $x$ as above, is equivalent to $a^{-1} x=x$. 
Moreover, by the second relator $M^pL^q$ of $G$, it follows that
\begin{align*}
 M^pL^q &= a^pL^q \\
&= a^{p-(s+t)(q-1)-t}La^{s+t}La^{s+t} \cdots La^{s+t}La^t \\
&= a^{p-(s+t)(q-1)-t}(a^{-s}wa^{-t})a^{s+t}(a^{-s}wa^{-t})a^{s+t} \cdots (a^{-s}wa^{-t})a^{s+t}(a^{-s}wa^{-t})a^t \\
&= a^{p-(s+t)(q-1)-t-s}w^q \\
&= a^{p-(s+t)q}w^q \\
&=1 , 
\end{align*}
for $a=M$ and $L$ commutes.  
Thus we have the relation $a^{p-(s+t)q} = w^{-q}$, and 
further that $a^{q(p-(s+t)q)} = w^{-q^2}$. 
On the other hand, we obtain $a^{-1} b^{-1} x<x$ by $a^{-1} b^{-1} x<a^{-1} x=x$. Repeating this consideration, we have $w^{-q^2}x<x$, for $q^{2} > 0$ and $w^{-1}$ is a word which contains at least one $b^{-1}$ and excludes $a$ and $b$. 
Therefore, we have 
\begin{align*}
 x &> w^{-q^2} x \\
&= a^{q(p-(s+t)q)}x \\
&= x
\end{align*}
giving a contradiction. Hence we have $bx=x$. 
Now $x$ is fixed by $a$ and $b$, and so $\Phi(G)$ has a global fixed point, since $G$ is generated by $a$ and $b$.

We next consider the case where $a$ does not have any fixed point on $\mathbb{R}$. 
We will show that this case cannot happen. 
In this case, we can suppose that $ax>x$ for any $x\in \mathbb{R}$ by reversing its order if necessary. 
Using the first relator of $G$, we obtain the following: 
\begin{align*}
 & w_1a^mw_1^{-1} b^{-r}w_2^{-1} a^nw_2b^{r-k}=1 \\
& w_1^{-1}b^{-r}w_2^{-1}a^nw_2b^{r-k}w_1a^m=1 \\
& a^nw_2b^{r-k}w_1a^m = w_2b^rw_1
\end{align*}
Then it follows that 
\begin{align*}
 a^nw_2b^{r-k}w_1a^mx 
&= w_2b^rw_1 x \\
&< w_2b^rw_1a^mx \\
&< a^nw_2b^rw_1a^mx
\end{align*}
since we are assuming that $ax>x$ for any $x\in \mathbb{R}$. 
Since any element of $G$ preserves an orientation of $\mathbb{R}$, 
we see that 
$
a^nw_2b^rw_1a^mx>a^nw_2b^{r-k}w_1a^mx$ implies $b^rw_1a^mx>b^{r-k}w_1a^mx$.
Since the element $w_1a^m\in G$ is thought as a homeomorphism of $\mathbb{R}$, there is a point $x\in \mathbb{R}$ which satisfies $x'=w_1a^mx$ for any point $x'\in \mathbb{R}$.
Hence, for the point $x'$, we obtain
\begin{align*}
b^rx' &= b^rw_1a^mx \\
&> b^{r-k}w_1a^mx \\
&= b^{r-k}x' \\
b^kx' &> x'
\end{align*}
Consequently it follows that $bx>x$ for any $x\in \mathbb{R}$ if $k>0$, and 
follows that $k \ne 0$ otherwise $x' > x'$. 

On the other hand, by using the second relator of $G$, 
by similar arguments as above, it must follow that 
$a^{q((s+t)q-p)}x>x$ by $a^{q((s+t)q-p)}x = w^{q^2} x > x$. 
This must imply that $q((s+t)q-p)>0$, that is, $p/q<s+t$, 
but this contradicts the assumption $p/q\ge s+t$. 

This completes the proof of the proposition. 
\end{proof}

\begin{proof}[Proof of Theorem~\ref{thm}.]
Suppose that $\pi_1(K(p/q))$ is left-orderable. 
It is known that a countable group $G$ is left-orderable if and only if $G$ is isomorphic with a subgroup of 
$Homeo^+(\mathbb{R})$.
See \cite[Theorem 2.6]{BRW}.
This implies that there exists an injective homomorphism 
$\Phi: \pi_1(K(p/q)) \to Homeo^+(\mathbb{R})$. 

It was also proved in \cite[Lemma 5.1]{BRW} that if there is a homomorphism $G \to Homeo^+(\mathbb{R})$ with image $\ne \{id\}$, 
then there is another such homomorphism which induces an action on $\mathbb{R}$ without global fixed points. 
Without loss of generality, we may assume that $\Phi$ has no global fixed point. This contradicts the previous proposition. 
Therefore $\pi_1(K(p/q))$ is non left-orderable.
\end{proof}

\section{Application}

In this section, we show that the knot groups for certain twisted torus knots admit the presentations satisfying the conditions in our theorem, and give a proof of Corollary \ref{cor}.  

A \textit{positively twisted torus knots} in $S^3$ is defined as the knots obtained from the torus knot by adding positive full twists along an adjacent pair of strands. 
Here we only consider positively $u$-twisted torus knots of type $(3, 3v+2)$, 
that is, the knots obtained from the torus knot of type $(3, 3v+2)$ by adding positive $u$ full twists along an adjacent pair of strands. 
Note that it is known that the positively $u$-twisted $(3,5)$-torus knot is equivalent to the $(-2, 3, 5 + 2u)$-pretzel knot.

\begin{proposition}\label{prp}
Let $K$ be the positively $u$-twisted $(3,3v+2)$-torus knot in $S^3$ with $u,v \ge 0$. 
Then the knot group $\pi_1(S^3-K)$ has a presentation such as 
\[
\langle a, b \ \vert \ (w_1 a^m w_1^{-1}) b^{-r} (w_2^{-1} a^n w_2) b^{r-k} \rangle
\]
with $m=n=1$, $r=u+1$, $k=1$, $w_1 = (ba)^{v+1}$, $w_2=(ba)^v$. 
Furthermore the generator $a$ represents a meridian of $K$ and 
the preferred longitude of $K$ is represented as $a^{-s} w a^{-t}$ with 
$s=2u+3(3v+2)+1$, $t=-1$ and $w = ((ba)^v b^{u+1} )^2 (ba)^v b$. 
\end{proposition}

\begin{proof}
As shown in \cite[Proposition 3.1]{CW}, 
$\pi_1(M-K)$ has the following presentation. 
\[
\langle x, y \ \vert \ x^2 ( y^{-v} x )^u x y^{-v-1} ( y^{-v} x )^{-u} y^{-2v-1} \rangle
\]
Moreover by \cite[Proposition 3.2]{CW}, 
the element $y^{v+1} x^{-1}$ represents a meridian of $K$. 
Thus we set $a = y^{v+1} x^{-1}$. 
Then, since $a = y^{v+1} x^{-1}$ is equivalent to $x = a^{-1} y^{v+1}$, we have 
\begin{align*}
  &  \langle x, y \ \vert \ x^2 ( y^{-v} x )^u x y^{-v-1} ( y^{-v} x)^{-u} y^{-2v-1} \rangle \\
&= \langle a, y \ \vert \ (a^{-1} y^{v+1} )^2 ( y^{-v} a^{-1} y^{v+1} )^u a^{-1} y^{v+1} y^{-v-1} (y^{-v} a^{-1} y^{v+1} )^{-u} y^{-2v-1} \rangle
\end{align*}
Let us transform the relator above as follows. 
\begin{align*}
 & (a^{-1} y^{v+1} )^2 (y^{-v} a^{-1} y^{v+1})^u a^{-1} y^{v+1} y^{-v-1} (y^{-v} a^{-1} y^{v+1})^{-u} y^{-2v-1} \\
&= ( a^{-1} y^{v+1} )( a^{-1} y^{v+1})
(y^{-v} a^{-1} y^{v+1})^u a^{-1} (y^{-v} a^{-1} y^{v+1})^{-u} y^{-2v-1} \\
&= ( a^{-1} y^{v+1} ) ( a^{-1} y^{v+1} )
( y^{-v} a^{-1} y^{v+1} ) 
\cdots (y^{-v} a^{-1} y^{v+1}) 
a^{-1} (y^{-v} a^{-1} y^{v+1})^{-u} y^{-2v-1} \\
&= ( a^{-1} y^{v+1} ) a^{-1} y a^{-1} y \cdots y a^{-1} y^{v+1} a^{-1}
 (y^{-v} a^{-1} y^{v+1} )^{-u} y^{-2v-1} \\
&= a^{-1} y^v ( y a^{-1} )^{u+1} y^{v+1} a^{-1} ( y^{-v} a^{-1} y^{v+1} )^{-u} y^{-2v-1} \\
&= a^{-1} y^v ( y a^{-1} )^{u+1} y^{v+1} a^{-1} ( y^{-v-1} a y^v )^u y^{-2v-1} \\
&= a^{-1} y^v ( y a^{-1})^{u+1} y^{v+1} a^{-1} ( y^{-v-1} a y^v) 
\cdots ( y^{-v-1} a y^v) y^{-2v-1} \\
&= a^{-1} y^v ( y a^{-1} )^{u+1} y^{v+1} a^{-1} y^{-v-1} a y^{-1} a y^{-1} \cdots a y^{-1} y^{-v} \\
&= a^{-1} y^v ( y a^{-1} )^{u+1} y^{v+1} a^{-1} y^{-v-1} ( a y^{-1} )^u y^{-v}\\
&= a^{-1} y^v ( y a^{-1} )^{u+1} y^{v+1} a^{-1} y^{-v-1} ( y a^{-1} )^{-u} y^{-v}
\end{align*}
Further, performing a cyclic permutation, we have the following. 
$$  ( y a^{-1})^{-u} y^{-v} a^{-1} y^v ( y a^{-1} )^{u+1} y^{v+1} a^{-1} y^{-v-1} $$
Now we set $b = y a^{-1}$. Substituting $y = b a$, we have the following. 
\[
b^{-u} ( b a )^{-v} a^{-1} ( b a )^v b^{u+1} ( b a )^{v+1} a^{-1} ( b a )^{-v-1} 
\]
Finally, taking the inverse of the above, we have the following relator. 
\[
( b a )^{v+1} a ( b a )^{-v-1} b^{-u-1} ( b a )^{-v} a ( b a )^v b^u
\]
Consequently the knot group $\pi_1(S^3-K)$ has a presentation such as 
\[
\langle a, b \ \vert \ (w_1 a^m w_1^{-1}) b^{-r} (w_2^{-1} a^n w_2) b^{r-k} \rangle
\]
with $m=n=1$, $r=u+1$, $k=1$, $w_1 = (ba)^{v+1}$, $w_2=(ba)^v$.

Also, by \cite[Proposition 3.2]{CW}, a longitude of $K$ is represented by 
$x^2 ( y^{-v} x )^u x ( y^{-v} x)^u $ in terms of the generators $x$ and $y$. 
Then, by substituting $x = a^{-1} y^{v+1} $, it is transformed as follows. 
\begin{align*}
 & x^2 ( y^{-v} x )^u x (y^{-v} x)^u \\
&= ( a^{-1} y^{v+1} )^2 (y^{-v} a^{-1} y^{v+1} )^u a^{-1} y^{v+1} (y^{-v} a^{-1} y^{v+1})^u \\
&= ( a^{-1} y^{v+1} ) ( a^{-1} y^{v+1} ) ( y^{-v} a^{-1} y^{v+1} )
\cdots ( y^{-v} a^{-1} y^{v+1} ) 
a^{-1} y^{v+1} ( y^{-v} a^{-1} y^{v+1})
\cdots ( y^{-v} a^{-1} y^{v+1}) \\
&= ( a^{-1} y^v ) ( y a^{-1})^{u+1} y^v ( y a^{-1} )^{u+1} y^{v+1}
\end{align*}
Again, by substituting $y = b a$, we obtain the following. 
\[
a^{-1} ( b a )^v b^{u+1} ( b a )^v b^{u+1} ( b a )^{v+1}
\]
Finally, as noted in the paragraph just below \cite[Proposition 3.2]{CW}, 
the presentation of the preferred longitude of $K$ is obtained 
from the above by adding $a^{-3(3v+2)-2u}$ as follows. 
\[
a^{-3(3v+2)-2u-1} ( ( ( b a )^v b^{u+1} )^2 ( b a )^v b ) a 
\]
Therefore the preferred longitude of $K$ is represented as $a^{-s} w a^{-t}$ with 
$s=2u+3(3v+2)+1$, $t=-1$ and $w = ((ba)^v b^{u+1} )^2 (ba)^v b$. 

This completes the proof of Proposition \ref{prp}. 
\end{proof}

Now Corollary \ref{cor} follows from Proposition \ref{prp} and Theorem \ref{thm}  immediately. 

\section*{Acknowledgements}
The authors would like to thank Kimihiko Motegi for useful discussions. 
They also thank Anh Tran and Yasuharu Nakae for their valuable comments, and the anonymous referee for his/her careful readings for the the earlier version of this paper. 
The first author is partially supported by JSPS KAKENHI Grant Number 26400100.


\begin{thebibliography}{0}

\bibitem{BS}
G. Baumslag\ and\ D. Solitar, Some two-generator one-relator non-Hopfian groups, {\it Bull. Amer. Math. Soc}. {\bf 68} (1962), 199--201. 
 
 \bibitem{B}
S. Boyer, Dehn surgery on knots, in {\it Handbook of geometric topology}, 165--218, North-Holland, Amsterdam. 

\bibitem{BGW}
S. Boyer, C. McA. Gordon\ and\ L. Watson, On L-spaces and left-orderable fundamental groups, Math. Ann. {\bf 356} (2013), no.~4, 1213--1245. 

\bibitem{BRW} 
S. Boyer, D. Rolfsen\ and\ B. Wiest, Orderable 3-manifold groups, {\it Ann. Inst. Fourier (Grenoble)} {\bf 55} (2005), no.~1, 243--288. 

\bibitem{CW}
A. Clay\ and\ L. Watson, Left-orderable fundamental groups and Dehn surgery, {\it Int. Math. Res. Not}. {\bf 2013}, no.~12, 2862--2890. 

\bibitem{N}
Y. Nakae, A good presentation of $(-2,3,2s+1)$-type pretzel knot group and $\Bbb R$-covered foliation, {\it J. Knot Theory Ramifications} {\bf 22} (2013), no.~1, 1250143, 23 pp. 

\bibitem{PZ}
P. Ozsv\'ath\ and\ Z. Szab\'o, On knot Floer homology and lens space surgeries, {\it Topology} {\bf 44} (2005), no.~6, 1281--1300.

\bibitem{OS}  
P. S. Ozsv\'ath\ and\ Z. Szab\'o, Knot Floer homology and rational surgeries, {\it Algebr. Geom. Topol}. {\bf 11} (2011), no.~1, 1--68. 

\bibitem{S}
P. B. Shalen, Three-manifolds and Baumslag-Solitar groups, {\it Topology Appl}. {\bf 110} (2001), no.~1, 113--118. 

\bibitem{V}
F. Vafaee, 
On the Knot Floer Homology of Twisted Torus Knots, 
to apear in Int Math Res Notices (2014). 


\end{thebibliography}
\end{document}